\documentclass[11pt]{amsart}
\usepackage{graphicx}
\usepackage{tikz}
\usepackage{pst-node}

\usepackage{amssymb}
\usepackage{epstopdf, xyling}
\newtheorem{thm}{Theorem}[section]
\newtheorem{prop}[thm]{Proposition}
\newtheorem{lem}[thm]{Lemma}

\newtheorem{qst}[thm]{Question}

\theoremstyle{definition}

\newtheorem{eg}[thm]{Example}

\theoremstyle{remark}

\newcommand{\Q}{\mathbb Q}
\newcommand{\C}{\mathbb C}
\newcommand{\ph}{\varphi}
\renewcommand{\P}{\mathbb P}

\renewcommand{\phi}{\varphi}
\newcommand{\im}{\mathop{\mathrm{Im}}\nolimits}
\renewcommand{\ker}{\mathop{\mathrm{Ker}}\nolimits}
\newcommand{\coker}{\mathop{\mathrm{Coker}}\nolimits}
\newcommand{\codim}{\mathop{\mathrm{codim}}\nolimits}

\newcommand{\iso}{\cong}

\renewcommand{\hat}{\protect\widehat}

\renewcommand{\i}[1]{\mathfrak{#1}}
\newcommand{\p}{\i{p}}
\newcommand{\q}{\i{q}}
\newcommand{\m}{\i{m}}

\newcommand{\F}{\mathcal F}

\newcommand{\x}{\mathbf x}

\newcommand{\rank}{\mathop{\mathrm{rank}}\nolimits}

\newcommand{\depth}{\mathop{\mathrm{depth}}\nolimits}
\newcommand{\grade}{\mathop{\mathrm{grade}}\nolimits}
\newcommand{\reg}{\mathop{\mathrm{reg}}\nolimits}

\newcommand{\ass}{\mathop{\mathrm{Ass}}\nolimits}
\newcommand{\h}{\mathop{\mathrm{ht}}\nolimits}

\newcommand{\supp}{\mathop{\mathrm{Supp}}\nolimits}
\newcommand{\spec}{\mathop{\mathrm{Spec}}\nolimits}

\newcommand{\pd}{\mathop{\mathrm{pd}}\nolimits}

\newcommand{\tor}{\mathop{\mathrm{Tor}}\nolimits}
\renewcommand{\hom}{\mathop{\mathrm{Hom}}\nolimits}
\newcommand{\ext}{\mathop{\mathrm{Ext}}\nolimits}

\renewcommand{\h}{\mathrm{ht}}

\title[Multiple Structures with Large Projective Dimension]{Multiple Structures with Arbitrarily Large Projective Dimension Supported on Linear Subspaces}

\author{Craig Huneke}
\address{University of Virginia, Department of Mathematics, 141 Cabell Drive, Kerchof Hall, P.O. Box 400137, Charlottesville, VA 22904-4137}
\email{huneke@virginia.edu}
\author{Paolo Mantero}
\address{University of California, Riverside, Department of Mathematics, 900 University Ave., Riverside, CA 92521}
\email{mantero@math.ucr.edu}
\author{Jason McCullough}
\address{Rider University, Department of Mathematics, 2083 Lawrenceville Road,
Lawrenceville, NJ 08648}
\email{jmccullough@rider.edu}
\author{Alexandra Seceleanu}
\address{University of Nebraska, Department of Mathematics, 203 Avery Hall, Lincoln, NE 68588}
\email{aseceleanu2@math.unl.edu}


\begin{document}

\begin{abstract} Let $K$ be an algebraically closed field.  There has been much interest in characterizing multiple structures in $\P^n_K$ defined on a linear subspace of small codimension under additional assumptions (e.g. Cohen-Macaulay).  We show that no such finite characterization of multiple structures is possible if one only assumes Serre's $(S_1)$ property holds.  Specifically, we prove that
for any positive integers $h, e \ge 2$ with $(h,e) \neq (2,2)$ and $p \ge 5$ there is a homogeneous ideal $I$ in a polynomial ring over $K$ such that $(1)$ the height of $I$ is $h$, $(2)$ the Hilbert-Samuel multiplicity of $R/I$ is $e$, $(3)$ the projective dimension of $R/I$ is at least $p$ and $(4)$ the ideal $I$ is primary to a linear prime $(x_1,\ldots, x_h)$.  This result is in stark contrast to Manolache's characterization of Cohen-Macaulay multiple structures in codimension $2$ and multiplicity at most $4$ and also to Engheta's characterization of unmixed ideals of height $2$ and multiplicity $2$.  \end{abstract}

\maketitle

\section{Introduction}

Let $K$ be an algebraically closed field. We consider projective multiple (i.e. generically nonreduced) schemes whose reduced subschemes are linear subspaces in $\P_K^n$ for some $n$.  Multiple structures in general have been widely studied with connections to vector bundles \cite{Hartshorne}, \cite{BF}, \cite{Manolache1}, Hartshorne's Conjecture \cite{VatneH}, linkage theory \cite{Migliore} and set-theoretic complete intersections \cite{Szpiro}.  In our setting where the reduced subscheme is a smaller projective space, there are finite characterizations of multiple structures in codimension two in small degree and under certain hypotheses:  Manolache gave structure theorems for scheme-theoretically Cohen-Macaulay
multiple structures of degree at most $4$ \cite{Manolache} and locally complete intersection multiple structures of degree at most $6$ \cite{ManolacheLCI}.  See \cite{ManolacheSurvey} for a nice survey of these results.

The defining ideals of these schemes correspond to homogeneous ideals that are primary to a prime ideal generated by linear forms in a polynomial ring $R$ over $K$.  More broadly, we were interested in the homological structure of homogeneous unmixed ideals of any polynomial ring over $K$, that is, ideals whose associated primes all have the same height.  

Engheta gave a complete characterization of unmixed ideals of height $2$ and multiplicity $2$:

\begin{prop}[{Engheta \cite[Prop. 11]{Engheta1}})]\label{Pe2} Let $R$ be a polynomial ring over an algebraically closed field and let $I \subset R$ be a height two unmixed ideal of multiplicity $2$.  Then $\pd(R/I) \le 3$ and $I$ is one of the following ideals.
\begin{enumerate}
\item A prime ideal generated by a linear form and an irreducible quadric.
\item $(x, y) \cap (x, z) = (x, yz)$ with independent linear forms $x, y, z$.
\item $(w, x) \cap (y, z) = (wy, wz, xy, xz)$ with independent linear forms $w, x, y, z$.
\item The $(x,y)$-primary ideal $(x, y)^2 + (ax + by)$ with independent linear forms $x,y$ and forms $a, b \in \m$ such that $x, y, a, b$ form a regular sequence.
\item $(x, y^2)$ with independent linear forms $x, y$.
\end{enumerate}
\end{prop}

 The hypothesis that $K$ is algebraically closed is essential.  Take for instance $R = \Q[w, x, y, z]$ and $P = (w^2 + x^2, y^2 + z^2, wz - xy, wy + xz)$.  Then $P$ is a prime ideal of height $2$ and multiplicity $2$, but is not degenerate (i.e. does not contain a linear form) as in case (1) above.  Note that over $\C$, $P \C[w, x, y, z]$ is no longer prime but rather of type $(3)$ since
\[ P\C[w, x, y, z] = (w + ix, y + iz) \cap (w - ix, y - iz).\]

One might wonder if there is a finite list for other multiple structures.  We show that this is hopeless is a very strong form.  Specifically, we give an explicit construction of homogeneous primary ideals of any other height and multiplicity and arbitrarily high projective dimension.  We state our main theorem here:

\begin{thm}\label{Tmain} Let $K$ be an algebraically closed field.  For any integers $h, e \ge 2$ with $(h, e) \neq (2,2)$ and for any integer $p \ge 5$, there exists an unmixed ideal $I_{h,e,p}$ of height $h$ and multiplicity $e$ in a polynomial ring $R$ over $K$ with $\sqrt{I_{h,e,p}}$ a linear prime and such that $\pd(R/I_{h,e,p}) \ge p$.  
\end{thm}

Our strategy is to produce three primary ideals of height $2$ and one of height $3$ in low multiplicity whose canonical modules have arbitrarily high projective dimension.  This construction relies on the Buchsbaum-Eisenbud acyclicity theorem \cite{BE} and on the existence of three-generated ideals with high projective dimension by Burch \cite{Burch} and Kohn \cite{Kohn}.  Indeed our result can be seen as closely related to the papers of Burch and Kohn but in a different direction. Finally, linkage arguments are used to produce the primary ideals in the main theorem.  

Our original motivation stems from the following open question first posed by Stillman:

\begin{qst}[{Stillman \cite[Problem 15.8]{PS}}]\label{Qstillman} Is there a bound, independent of n, on the projective dimension of ideals in $R = K[x_1, . . . , x_n]$ which are generated by $N$ homogeneous
polynomials of given degrees $d_1,\ldots, d_N$?
\end{qst}

In light of Bruns's Theorem \cite{Bruns}, the most significant case is that of a three-generated ideal $(f,g,h)$.  The answer is clear when the height of $(f,g,h)$ is $1$ or $3$, so we can assume the height is $2$ and that $f, g$ form a regular sequence.  In this case, we have a short exact sequence
\[0 \to \frac{R}{(f, g):h} \xrightarrow{h} \frac{R}{(f, g)} \to \frac{R}{(f, g, h)} \to 0.\]
Since $\pd(R/(f, g)) = 2$, it suffices to bound the projective dimension of the left-hand term.  The ideal $(f, g):h$ is unmixed of height $2$ and multiplicity at most $\deg(f) \cdot \deg(g)$.  This argument reduces the problem to bounding the projective dimension of unmixed ideals in terms of their height and multiplicity.  Theorem~\ref{Tmain} shows that this is impossible in general, even in height $2$.   This answers Question~6.4 raised by two of the authors in \cite{MS} negatively.

We also remark that Caviglia \cite{Caviglia} showed that Question~\ref{Qstillman} is equivalent to the question in which we replace projective dimension with regularity.  One could use a similar strategy and try to bound the regularity of unmixed ideals of a fixed height and multiplicity.  Proposition~\ref{Pe2} already shows this is fruitless.  The ideals $J_n = (x^2, xy, y^2, w^nx+z^ny)$ are unmixed of height $2$ and multiplicity $2$ and satisfy $\reg(R/J_n) = n$.

The rest of this paper is structured as follows: Section~\ref{Sb} contains notation and basic results needed for the remainder of the paper.  In Section~\ref{Smr} we prove the main theorem while relegating the technical details about the construction of the four specific ideals mentioned above to Section~\ref{S4ideals}.  This construction is summarized in Proposition~\ref{P4ideals}.  In the final Section~\ref{Seq} we construct a specific example from our family of primary ideals with large projective dimension and discuss some remaining questions.

\section{Background}\label{Sb}

For the rest of this paper, $R$ will denote a polynomial ring over an algebraically closed field $K$.   We consider $R$ as a standard graded ring.  We write $R_i$ for the $K$-vector space of homogeneous degree $i$ polynomials of $R$.  For a finitely generated graded $R$-module $M$, there exists a unique (up to isomorphism) minimal graded free resolution
\[F_0 \xleftarrow{d_1} F_1 \xleftarrow{d_2} \cdots \xleftarrow{d_p} F_p \leftarrow 0,\]
that is, an exact sequence of graded maps of finitely generated graded free modules $F_i = \bigoplus_j R(-j)^{\beta_{i,j}(M)}$, where $M \iso \coker(d_1)$.  Here $R(-d)$ denotes a rank one free module with generator in degree $d$ so that $R(-d)_i = R_{i-d}$.  The numbers $\beta_{i,j}(M)$ are invariants of $M$ called the \textit{graded Betti numbers} of $M$ and can also be defined as $\beta_{i,j}(M) = \tor_i^R(M,K)_j$.  The length $p$ of the minimal free resolution of $M$ is called the \textit{projective dimension} of $M$ and is denote $\pd(M)$.  By convention, we often write the Betti numbers of $M$ as a matrix called the $\textit{Betti table}$ of $M$:
$$\begin{tabular}{c|cccccc}
       &0&1&2&$\cdots$&i&$\cdots$\\ \hline
         \text{0:}&$\beta_{0,0}(M)$&$\beta_{1,1}(M)$&$\beta_{2,2}(M)$&$\cdots$&$\beta_{i,i}(M)$&$\cdots$\\
         \text{1:}&$\beta_{0,1}(M)$&$\beta_{1,2}(M)$&$\beta_{2,3}(M)$&$\cdots$&$\beta_{i,i+1}(M)$&$\cdots$\\
         \text{2:}&$\beta_{0,2}(M)$&$\beta_{1,3}(M)$&$\beta_{2,4}(M)$&$\cdots$&$\beta_{i,i+2}(M)$&$\cdots$\\
  $\vdots$&$\vdots$&$\vdots$&$\vdots$&&$\vdots$&\\
  \text{j:}&$\beta_{0,j}(M)$&$\beta_{1,j+1}(M)$&$\beta_{2,j+2}(M)$&$\cdots$&$\beta_{i,i+j}(M)$&$\cdots$\\
    $\vdots$&$\vdots$&$\vdots$&$\vdots$&&$\vdots$&\\
    \end{tabular}$$
    The projective dimension of $M$ is then the index of the last nonzero column in the Betti table of $M$.

The following lemma is useful when computing projective dimension. 

\begin{prop}\label{Ppd} Let $0 \to A \to B \to C \to 0$ be a short exact sequence of $R$-modules.  Then
\begin{enumerate}
\item $\pd(A) \le \max\{\pd(B), \pd(C) - 1\}$,
\item $\pd(B) \le \max\{\pd(A), \pd(C)\}$,
\item $\pd(C) \le \max\{\pd(A) + 1, \pd(B)\}$.
\end{enumerate}
\end{prop}

For an ideal $I$, the \textit{unmixed part} of $I$ is the intersection of primary components of $I$ corresponding to primes of minimal height:
\[I^{un} = \bigcap_{\substack{\p \in \ass(I)\\\h(\p) = \h(I)}} \q_\p, \]
where $\q_\p$ is the $\p$-primary component of $I$.
An ideal $I$ is \textit{unmixed} if $I = I^{un}$.  By way of the associativity formula, we have a way of characterizing many unmixed ideals.  We denote by $\lambda(-)$ the length of a module and by $e(-)$ the Hilbert-Samuel multiplicity.

\begin{thm}[Associativity Formula (cf. {\cite[Thm 11.2.4]{HS}})] Let $I$ be an ideal of $R$.  Then
\[e(R/I) = \sum_{\substack{\p \in \spec(R)\\\h(\p) = \h(I)}} e(R/\p) \lambda(R_\p/I_\p).\]
\end{thm}

It follows that $e(R/I) = e(R/I^{un})$.  


Let $\depth(M)$ denote the length of a maximal regular sequence on $M$.  We say that $M$ satisfies \textit{Serre's condition $(S_k)$} (or simply is $(S_k)$) if
\[\depth(M_\p) \ge \min_{\p \in \supp(M)}\{k,\dim(M_\p)\}.\]
An ideal $I$ is unmixed if and only if $R/I$ satisfies $(S_1)$.

Let $\ph:F \to G$ be a map between finite rank free module $F$ and $G$.  After choosing bases for $F$ and $G$, we can represent $\ph$ be a matrix.  For a positive integer $j$ we denote by $I_j(\ph)$ the ideal of $j \times j$ minors of the entries in the matrix representing $\ph$.  Note that $I_j(\ph)$ does not depend on the choice of bases (cf. \cite[p. 497]{Eisenbud1}).

\begin{thm}[Buchsbaum-Eisenbud  (cf. {\cite[Thm. 20.9]{Eisenbud1}})]\label{Tres} Let $\F$ be a complex free $R$-modules of finite rank
\[\F: F_0 \xleftarrow{d_1} F_1 \xleftarrow{d_2} F_2 \xleftarrow{d_3} \cdots \xleftarrow{d_p} F_p \leftarrow 0.\]
Set $r_j = \sum_{i = j}^p (-1)^{p-i} \rank F_i$.  Then $\F$ is a resolution of $M := \coker(d_1)$ if and only if 
\[\h(I_{r_j}(d_j)) \ge j \qquad \text{for all $j = 1,\ldots, p$.}\]
\end{thm}
Note that in our setting where $R$ is a polynomial ring, $\h(I) = \grade(I)$ for an ideal $I$.  We prefer to work with height but the more general statement of the above theorem involves the grades of ideals of minors.

The following result seems to be well-known, but we sketch a proof for completeness.

\begin{prop}\label{Csk} Using the notation from the previous theorem, suppose $\F$ is a minimal free resolution of $M$.  Then $M$ satisfies Serre's condition $(S_k)$ if and only if
\[\h(I_{r_j}(d_j)) \ge \min\{\dim(R), j+k\} \qquad \text{for all $j = \codim(M)+1,\ldots, p$.}\]
\end{prop}

\begin{proof} We argue by contrapositive.  
\begin{align*}
&M \text{ is not } S_k  \Leftrightarrow  \exists\, \p \in \supp(M) \text{ with } \depth(M_\p) < \min\{k, \dim(M_\p)\}\\
& \Leftrightarrow \exists\, \p \in \supp(M)\text{ with } \pd_{R_\p}(M_\p) > \max\{\h(\p) - k, \h(\p) - \dim(M_\p)\}\\
& \Leftrightarrow \exists\, \p \in \supp(M)\text{ with } \pd_{R_\p}(M_\p) > \max\{\h(\p) - k, \codim(M)\}\\
&  \Leftrightarrow \exists\, \p \in \supp(M)\text{ with }\p \supset I_{r_{i+1}}(d_{i+1}) \text{ and } i = \max\{\h(\p) - k, \codim(M)\}\\
& \Leftrightarrow \h(I_{r_{i+1}}(d_{i+1})) \le \min\{\dim(R),i + k\} \text{ for some } i \ge \codim(M). 
\end{align*}
For the second equivalence, we use the Auslander-Buchsbaum Theorem.  The third equivalence uses that fact that $\codim(M) = \codim(M_\p)$ for all $\p \in \supp(M)$.  The fourth equivalence follows because if $\p \not\supset I_{r_{i+1}}(d_{i+1})$, then $d_{i+1}$ splits over $R_\p$ and hence $\pd_{R_\p}(M_\p) \le i$.  Otherwise, $d_{i+1}$ remains minimal over $R_\p$ and $\pd_{R_\p}(M_\p) > i$.
\end{proof}

Two unmixed ideals $I$, $J$ are \textit{linked} via the complete intersection $(\x)$, where $\x = x_1,\ldots, x_h \in I \cap J$, if $I = (\x):J$ and $J = (\x):I$.  Moreover, if $x_1,\ldots, x_h$ are homogeneous elements of degrees $d_1,\ldots, d_h$, respectively, then $e(R/I) + e(R/J) = e(R/(\x)) = \prod_{i = 1}^h d_i$.  If $I$ and $I'$ are linked to the same ideal $J$, then they share many properties.  In particular, we will make use of the following well-known fact.  (cf. \cite[Lemma 2.6]{EnghetaT}.)

\begin{prop}\label{P2link} Suppose $I$ and $I'$ are unmixed ideals of $R$ linked to the same ideal $J$.  Then
\[\pd(R/I) = \pd(R/I').\]
\end{prop}

The following lemma will be useful in proving certain ideals have large projective dimension.

\begin{lem}\label{Lasspd} Let $M$ be a finitely generated $R$-module.     Then 
\[\pd(M) \ge \max\{\h(\p)\,:\,\p \in \ass(M)\}.\]
\end{lem}

\begin{proof} We have 
\begin{eqnarray*}
\p \in \ass_R(M) &\Leftrightarrow& \p R_\p \in \ass_{R_\p}(M_\p)\\
 &\Leftrightarrow& \depth_{\p R_\p}(M_\p) = 0\\
  &\Leftrightarrow& \pd_{R_\p}(M_\p) = \mathrm{depth}_{\p R_\p}(R_\p) \\
  &\Rightarrow& \pd_R(M) \ge \mathrm{depth}_{\p R_\p}(R_\p) = \dim(R_\p) = \h(\p),\\
  \end{eqnarray*}
   where the third implication follows from the Auslander-Buchsbaum theorem, and the fourth implication follows from the flatness of $R_\p$ over $R$ and the fact that $R_\p$ is regular.
   \end{proof}
Note that we cannot use this lemma directly in our construction of primary ideals of large projective dimension since, by definition, they have no associated primes with large height.

There are now many constructions of three-generated ideals with large projective dimension \cite{Burch}, \cite{Kohn}, \cite{Bruns},  \cite{BMNSSS}.  We will use the following construction from \cite{McCullough} which gives ideals generated by three homogeneous elements of degree $n$ and projective dimension $n + 2$.

\begin{prop}\label{P3gen} Let $f_n = a^n, g_n = b^n, h_n = a^{n-1}c_1 + a^{n-2}bc_2 + \cdots + b^{n-1}c_n \in R = K[a, b, c_1,\ldots, c_n]$.  Let $\m$ denote the graded maximal ideal.  Then $\pd(R/(f_n, g_n, h_n)) = n+2$ and $a^{n-1}b^{n-1} \in (I:\m) \setminus I$.
\end{prop}

\begin{proof} That $s = a^{n-1}b^{n-1} \notin I$ is clear since none of the terms of $f_n, g_n, h_n$ divide $s$.  One checks that $sa, sb, sc_i \in (f_n, g_n, h_n)$ for all $i = 1,\ldots,n$.  Hence one has $\m \in \ass(R/I)$ and, by the previous lemma, $\pd(R/I) = \h(\m) = n + 2$.  
\end{proof}

Finally, we say that a finitely-generated $A$-module $\omega_A$ is a \textit{canonical module} for a graded ring $A$ if $\hat{\omega}_A \iso \mathrm{H}_\m^{\mathrm{dim}(A)}(A)^\vee$, where $(-)^\vee$ denotes the Matlis dual, $\hat{\phantom{--}}$ denotes $\m$-adic completion and $\mathrm{H}^i_\m(-)$ denotes the $i$th local cohomology module with respect to the graded maximal ideal.  When $A = R/I$, where $R = K[x_1,\ldots, x_n]$, then we can identify $\omega_A = \ext^{\h(I)}_R(R/I, R)$.  Note that when $A$ is not Cohen-Macaulay, some of the usual properties of the canonical module do not hold (e.g. the injective dimension of $\omega_A$ is not finite), but we only need the fact that $\omega_{R/I}$ can be written in terms of an ideal $J$ linked to $I$ (Lemma~\ref{canonicalModule}).

\section{Main Results}\label{Smr}


\subsection{The method of proof} 

Let us briefly describe our construction of primary ideals of large projective dimension.  The first step is to define four families of primary ideals with well-behaved resolutions and canonical modules with large projective dimension.  We seem to need four such families, $L_{2,5,p}, L_{2,6,p}, L_{2,20,p}$ and $L_{3,6,p}$, and relegate the details of those constructions to Section~\ref{S4ideals}.  We adopt the convention that $L_{h,e,p}$ denotes an ideal in a polynomial ring $R$ over $K$ such that $\h(L_{h,e,p}) = h$, $e(R/L_{h,e,p}) = e$ and $\pd(\omega_{R/L_{h,e,p}}) \ge p$.  

Linking via a complete intersection (sometimes more than once) from one of these base cases, we can produce primary and radical linear ideals of height $2$ and any multiplicity $3$ or larger, with arbitrarily high projective dimension.  Appending extra linear generators gives us examples with arbitrary height.  In addition to the three families of height $2$ ideals, we need one more construction for the height $3$ multiplicity $2$ case, since the height $2$ multiplicity $2$ case is finite.  In general, the ideals we construct have many generators.  Hence, we found that it is easier to work with the ideals to which they are linked.  We work out in detail one example in Section~\ref{Seq}.

\subsection{Homological preliminaries}

Here we collect a few results we use to connect the four families of ideals from Section~\ref{S4ideals} to the ones in the main theorem.  First we recall that the canonical module of an unmixed ideal $L$ can be written in terms of any ideal linked to it.

\begin{lem}\label{canonicalModule}
If $L$ is an ideal of height $h$ in a polynomial ring, ${\bf x}$ is a regular sequence of length $h$ contained in $L$ , then $\ext^h_R(R/L,R) \simeq (({\bf x}):L)/({\bf x}) $.
\end{lem}
\begin{proof}
By \cite[Lemma 1.2.4]{BH}, $\ext^h_R(R/L,R) \simeq \hom(R/L,R/({\bf x}))$.  The latter module is isomorphic to  $(({\bf x}):L)/({\bf x})$.
\end{proof}

\begin{prop}\label{Plink}
Let $L$ be an unmixed ideal of height $h$, let $({\bf x})$ be a complete intersection ideal of height $h$ contained in $L$ and let $I=({\bf x}):L$. If   $\pd(\ext^h_R(R/L,R))\geq h+1$, 
then 
\[\pd(R/I) = \pd(\ext^h_R(R/L,R)) + 1.\]  
\end{prop}

\begin{proof}
We have the short exact sequence
$$0 \to(({\bf x}):L)/({\bf x}) \to R/({\bf x}) \to R/(({\bf x}):L) \to 0.$$
By Proposition~\ref{Ppd},  $\pd (R/(({\bf x}):L)) = \pd ((({\bf x}):L)/({\bf x}))+1$, 
as long as  $\pd ((({\bf x}):L)/({\bf x})) \geq\pd( R/({\bf x})) + 1= h + 1$. By Lemma~\ref{canonicalModule}, $ (({\bf x}):L)/({\bf x}) \simeq \ext^h_R(R/L,R)$, so the above statement reads $\pd(R/I)=\pd(\ext^h_R(R/L,R)) + 1$ as long as $\pd(\ext^h_R(R/L,R))\geq h + 1$.
\end{proof}

In order to compute the projective dimension of $Ext^h_R(R/L,R)$, we begin by analyzing $h^{th}$ boundaries of a dualized free resolution of $R/L$.

\begin{lem}\label{Bsmall}
Let $R$ be a polynomial ring over $K$. Let $L$ be an ideal of $R$ of height $h$ and let 
\[F_0 \xleftarrow{d_1} F_1 \xleftarrow{d_2} \cdots \xleftarrow{d_p} F_p \leftarrow 0\]
be the minimal free resolution of $R/L$. Let $(-)^\ast$ denore the functor $Hom_R( -,R)$. Then
\[\pd(\im(d_h^\ast)) = h - 1.\]
\end{lem}

\begin{proof} Applying $\hom(-,R)$ to  the resolution of $R/L$ gives us the complex
$$ 0 \to F_0^\ast  \stackrel{d_1^\ast}\to F_1^\ast  \stackrel{d_2^\ast}\to \ldots \stackrel{d_{h-2}^\ast}\to F_{h-2}^\ast \stackrel{d_{h-1}^\ast}\to F_{h-1}^\ast \to \coker(d_{h-1}^\ast) \to 0.$$
Since $\h(L) = h$, $\ext_R^i(R/L,R) = 0$ for $0 \le i \le h - 1$.  Hence the above complex is exact.  Moreover, since $\ext_R^{h-1}(R/L,R) = 0$, 
\[\coker(d_{h-1}^\ast) = F_{h-1}^\ast/\im(d_{h-1}^\ast) = 
F_{h-1}^\ast/\ker(d_{h}^\ast) \iso \im(d_h^\ast).\]  
Hence $\pd(\im(d_h^\ast)) = h-1$.


 


\end{proof}

\begin{prop}\label{Pextker}
Let $L$ be an unmixed ideal of height $h$ and let $({\bf x})$ be a complete intersection ideal of height $h$ contained in $L$. Let $d_{h+1}$ denote the $(h+1)^{th}$ differential in the minimal resolution of $R/L$. If  $\pd(\ker(d_{h+1}^\ast)) \ge h$, then $\pd(\ext^h_R(R/L,R)) =\pd(\ker(d_{h+1}^\ast))$.
\end{prop}

\begin{proof}
We write 
\[F_0 \xleftarrow{d_1} F_1 \xleftarrow{d_2} \cdots \xleftarrow{d_p} F_p \leftarrow 0\]
 for the minimal free resolution of $R/L$ and we consider the short exact sequence defining $\ext^h_R(R/L,R)$:
$$0\to B \to Z \to \ext^h_R(R/L,R)\to 0,$$
where $B=\im(d_h^\ast)$ and $Z=\ker(d_{h+1}^\ast)$.

By Lemma~\ref{Bsmall}, $\pd(B) = h-1$. By Prop~\ref{Ppd}, $\pd(\ext^h_R(R/L,R)) = \pd(Z)$, whenever $\pd(Z) \ge \pd(B) +1= h$.
\end{proof}

\subsection{The base cases}

Here we list only those properties of the four families of ideals that we need to complete the proof of the main theorem.  Refer to Section~\ref{S4ideals} for details.

\begin{prop}\label{P4ideals} For every integer $p \ge 4$, there exist ideals $L_{2, 6, p}, L_{2,5,p}, L_{2,20,p}, L_{3,6,p}$, each homogeneous in a polynomial ring $R$ over $K$ such that
\begin{enumerate}
\item $L_{2,5,p}$ is primary to $(x,y)$ for independent linear forms $x,y$, $e(R/L_{2,5,p}) = 5$, $\pd(\ext^2_R(R/L_{2,5,p},R)) \ge p$, and $x^3, y^3 \in L_{2,5,p}$.
\item $L_{2,6,p}$ is primary to $(x,y)$ for independent linear forms $x,y$, $e(R/L_{2,6,p}) = 6$, $\pd(\ext^2_R(R/L_{2,6,p},R)) \ge p$, and $x^3, y^3 \in L_{2,6,p}$.
\item $L_{2,20,p}$ is primary to $(x,y)$ for independent linear forms $x,y$, $e(R/L_{2,20,p}) = 20$, $\pd(\ext^2_R(R/L_{2,20,p},R)) \ge p$, and $x^5, y^5 \in L_{2,20,p}$.
\item $L_{3,6,p}$ is primary to $(x,y,z)$ for independent linear forms $x,y,z$, $e(R/L_{3,6,p}) = 6$, $\pd(\ext^3_R(R/L_{3,6,p},R)) \ge p$, and $x^2, y^2, z^2 \in L_{3,6,p}$.
\end{enumerate}
\end{prop}

\begin{proof}
Let $L = L_{2,5,p}$ from Section~\ref{Sl25p} for a fixed integer $p \ge 4$.  By Proposition~\ref{Pl25p}, $L$ has height $2$, multiplicity $5$ and $\pd(\ker(d_3^\ast)) \ge p$, where $d_3$ is the third differential in the resolution of $R/L$.  By Proposition~\ref{Pextker}, $\pd(\ext_R^2(R/L,R)) \ge p$.  

The other three cases are identical.  See Section~\ref{S4ideals}.
\end{proof}

\subsection{Proof of the Main Theorem}

Using the base families in Proposition~\ref{P4ideals}, we can now finish off the proof of the main theorem.

\begin{proof}[proof of Theorem~\ref{Tmain}]
We first construct $I_{2,e,p}$ for $e \ge 3$ and $p \ge 5$.  We make the following assignments:
\begin{eqnarray*}
I_{2,3,p} &:=& (x^3, y^3):L_{2,6,p-1},\\
I_{2,4,p} &:=& (x^3, y^3):L_{2,5,p-1},\\
I_{2,5,p} &:=& (x^5, y^5):L_{2,20,p-1},\\
I_{2,6,p} &:=& (x^4, y^3):L_{2,6,p-1},\\
L_{2,8,p-1} &:=& (x^4,y^3):I_{2,4,p},\\
L_{2,11,p-1} &:=& (x^5,y^3):I_{2,4,p}\\
L_{2,14,p-1} &:=& (x^6,y^3):I_{2,4,p}\\
\end{eqnarray*}
and for integers $n \ge 0$, we set
\begin{eqnarray*}
I_{2,7+4n,p} &:=& (x^4,y^{3+n}):L_{2,5,p-1}\\
I_{2,8+4n,p} &:=& (x^{4+n},y^4):L_{2,8,p-1}\\
I_{2,9+4n,p} &:=& (x^{5+n},y^4):L_{2,11,p-1}\\
I_{2,10+4n,p} &:=& (x^{6+n},y^4):L_{2,14,p-1}\\
\end{eqnarray*}
That $I_{h,e,p}$ is $(x,y)$-primary and has the claimed multiplicity follows from the definitions and Proposition~\ref{P4ideals}.  That $\pd(R/I_{2,e,p}) \ge p$ follows from Proposition~\ref{Plink} and Proposition~\ref{P2link}.

We also define
\begin{eqnarray*}
I_{3,2,p} &:=& (x^2, y^2, z^2):L_{3,6,p-1}\\
\end{eqnarray*}

For $h \ge 3$ and $e \ge 3$, we set $I_{h,e,p} = I_{2,e,p} + (z_1,\ldots, z_{h - 2})$, where $z_1,\ldots, z_{h-1}$ are new variables added to the ambient ring of $I_{2,e,p}$.  We note that $\pd(I_{h,e,p}) = \pd(I_{2,e,p}) + h - 2$ since the variables of the two defining ideals are disjoint (i.e. the resolution of $R/I_{h,e,p}$ is just the tensor product of the resolutions of $I_{2,e,p}$ and $(z_1,\ldots, z_{h-2})$).  Note that $I_{h,e,p}$ is now primary to $(x, y, z_1,\ldots, z_{h-2})$ and hence unmixed of the proper height and multiplicity.

Similarly, for $h \ge 4$, we set $I_{h,2,p} = I_{3,2,p} + (z_1,\ldots, z_{h-3})$ for new variables $z_1,\ldots, z_{h-3}$.  This completes the proof.
\end{proof}

The linkage structure of the definitions from the previous proof is pictured in the diagram below.  The ideals in bold face are the four base cases from Section~\ref{S4ideals}.

\begin{center}
\begin{tikzpicture}
\usetikzlibrary{patterns}
\draw  [shape=circle] (-1,8) circle (.8) node {$\boldsymbol{L_{2,5,p-1}}$};
\draw  [shape=circle] (-1,6) circle (.8) node {$L_{2,8,p-1}$};
\draw  [shape=circle] (-1,4) circle (.8) node {$L_{2,11,p-1}$};
\draw  [shape=circle] (-1,2) circle (.8) node {$L_{2,14,p-1}$};
\draw  [shape=circle] (4,9) circle (.8) node {$I_{2,7 + 4n,p}$};
\draw  [shape=circle] (4,7) circle (.8) node {$I_{2,4,p}$};
\draw  [shape=circle] (6,5) circle (.8) node {$I_{2,8 + 4n,p}$};
\draw  [shape=circle] (6,3) circle (.8) node {$I_{2,9 + 4n,p}$};
\draw  [shape=circle] (6,1) circle (.8) node {$I_{2,10 + 4n,p}$};
\draw [line width=1pt  ] (-.2,8)--(3.2,9)  node [pos=.56, above ] { $x^4,y^{n+3}$ }; 
\draw [line width=1pt  ] (-.2,8)--(3.2,7)  node [pos=.5, above ] { \,\,\,\,$x^3,y^{3}$ }; 
\draw [line width=1pt  ] (-.2,6)--(3.2,7)  node [pos=.4, above ] { $x^4,y^{3}$ }; 
\draw [line width=1pt  ] (-.2,4)--(3.2,7)  node [pos=.3, left ] { $x^5,y^{3}$ }; 
\draw [line width=1pt  ] (-.2,2)--(3.2,7)  node [pos=.2, left ] { $x^6,y^{3}$ }; 
\draw [line width=1pt  ] (-.2,2)--(5.2,1)  node [pos=.8, above ] { \,\,$x^{6+n},y^{4}$ }; 
\draw [line width=1pt  ] (-.2,4)--(5.2,3)  node [pos=.8, above ] { \,\,$x^{5+n},y^{4}$ }; 
\draw [line width=1pt  ] (-.2,6)--(5.2,5)  node [pos=.8, above ] { \,\,$x^{4+n},y^{4}$ }; 
\end{tikzpicture}
\end{center}

\begin{center}
\begin{tikzpicture}
\usetikzlibrary{patterns}
\draw  [shape=circle] (-1,8) circle (.8) node {$\boldsymbol{L_{2,6,p-1}}$};
\draw  [shape=circle] (4,9) circle (.8) node {$I_{2,3,p}$};
\draw  [shape=circle] (4,7) circle (.8) node {$I_{2,6,p}$};
\draw [line width=1pt  ] (-.2,8)--(3.2,9)  node [pos=.56, above ] { $x^3,y^{3}$ }; 
\draw [line width=1pt  ] (-.2,8)--(3.2,7)  node [pos=.56, above ] { $x^3,y^{4}$ }; 
\end{tikzpicture}
\end{center}

\begin{center}
\begin{tikzpicture}
\usetikzlibrary{patterns}
\draw  [shape=circle] (-1,8) circle (.9) node {$\boldsymbol{L_{2,20,p-1}}$};
\draw  [shape=circle] (4,8) circle (.8) node {$I_{2,5,p}$};
\draw [line width=1pt  ] (-.1,8)--(3.2,8)  node [pos=.56, above ] { $x^5,y^{5}$ }; 
\end{tikzpicture}
\end{center}

\begin{center}
\begin{tikzpicture}
\usetikzlibrary{patterns}
\draw  [shape=circle] (-1,8) circle (.8) node {$\boldsymbol{L_{3,6,p-1}}$};
\draw  [shape=circle] (4,8) circle (.8) node {$I_{3,2,p}$};
\draw [line width=1pt  ] (-.2,8)--(3.2,8)  node [pos=.56, above ] { $x^2,y^{2},z^2$ }; 
\end{tikzpicture}
\end{center}

\section{Four Families of Primary Ideals and their Resolutions}\label{S4ideals}

In this section we construct the four families of ideals listed in Proposition~\ref{P4ideals}.  The technique is the same in each case.  We construct a primary ideal $L$ of height $c$ whose generators are defined in terms of forms $f, g, h$ of a fixed degree and with some height restrictions.  We show that the resolution of $R/L$ does not depend on the choice of $f, g, h$, appealing to Theorem~\ref{Tres} and Proposition~\ref{Csk}.  We then connect the projective dimension of the canonical module $\omega_{R/L} = \ext_R^h(R/L,R)$ with the projective dimension of the ideal $(f, g, h)$ by studying the maps in the resolution of $R/L$.  

\subsection{Construction of $L_{2,5,p}$}\label{Sl25p}

\begin{prop} \label{Pl25p}
Let $R$ be a polynomial ring over $K$ and let $x,y$ be independent linear forms in $R$.  Suppose $f, g, h \in R_d$ for some $d \ge 1$ such that $\mathrm{ht}(x,y,f,g,h) \ge 4$.  Let
\[
L = (x,y)^3 + (y^2f + xyg + x^2h).\]
Then $R/L$ has the following free resolution:      
\[R \xleftarrow{d_1} R^5 \xleftarrow{d_2} R^5 \xleftarrow{d_3} R \leftarrow 0,\]
      where
      \[d_1 = \begin{pmatrix}x^{3}&
     x^{2} y&
     x y^{2}&
     y^{3}&
     y^{2} f+x y g+x^{2} h\\
     \end{pmatrix},\]
     \[d_2 = \begin{pmatrix}{-y}&
     0&
     0&
     {-h}&
     0\\
     x&
     {-y}&
     0&
     {-g}&
     {-h}\\
     0&
     x&
     {-y}&
     {-f}&
     {-g}\\
     0&
     0&
     x&
     0&
     {-f}\\
     0&
     0&
     0&
     x&
     y\\
     \end{pmatrix},\]
     \[\text{and } d_3 = \begin{pmatrix}h\\
     g\\
     f\\
     {-y}\\
     x\\
     \end{pmatrix}.\]
Moreover, 
\begin{enumerate}
\item $L$ is $(x,y)$-primary, and
\item $e(R/L) = 5$.
\item $\pd(\ker(d_3^\ast)) = \pd(R/(x,y,f,g,h)) - 2$.
\end{enumerate}
\end{prop}

\begin{proof} It is easy to check the the above sequence is a complex and that $x^3, y^3 \in I_1(d_1)$, $x^4, y^4 \in I_4(d_2)$ and $I_1(d_3) = (x, y, f, g, h)$.  By Theorem~\ref{Tres}, this is a resolution of $R/L$, and by Proposition~\ref{Csk}, $L$ is unmixed.

Since $\sqrt{L} = (x,y)$, $(x,y)$ is the unique minimal prime of $L$.  Localizing at $\p = (x,y)$ we see that $\lambda(R_\p/L_\p) = 5$.  Hence $e(R/L) = 5$ by the associativity formula.  

Clearly $\im(d_3^\ast) = (x, y, f, g, h)$.  So $\pd(\ker(d_3^\ast)) = \pd(R/(x,y,f,g,h)) - 2$.
\end{proof}

We can now define $L_{2,5,p}$ to be any ideal in the polynomial ring by taking $f, g, h$ in the previous proposition to be forms of a fixed degree in variables disjoint from $x$ and $y$ with $\h(f, g, h) = 2$ and $\pd(R/(f, g, h)) = p$.  The forms $f_p, g_p, h_p$ in Proposition~\ref{P3gen} are one such choice.   It follows that $\pd(\ker(d_3^\ast)) = p$.

\subsection{Construction of $L_{3,6,p}$}\label{Sl36p}

\begin{prop}\label{Pl36p}
 Let $R$ be a polynomial ring over $K$ and let $x,y,z$ be independent linear forms in $R$.  Let $f, g, h \in R_d$ such that $\mathrm{ht}(x,y,z,f,g,h) \ge 5$.  Let
\[
L = (x^2, y^2, z^2, xyz,  xyh + xzg + yzf).\]
Then $R/L$ has the following free resolution
\[R \xleftarrow{d_1} R^5 \xleftarrow{d_2} R^9 \xleftarrow{d_3} R^6 \xleftarrow{d_4} R \leftarrow 0,\] 
where
\[d_1 = \begin{pmatrix}x^{2}&
      y^{2}&
      z^{2}&
      x y z&
      y z f+x z g+x y h\\
      \end{pmatrix},\]
      \[d_2 = \begin{pmatrix}{-y^{2}}&
      {-y z}&
      0&
      {-z^{2}}&
      0&
      0&
      -z g-y h&
      0&
      0\\
      x^{2}&
      0&
      {-x z}&
      0&
      0&
      {-z^{2}}&
      0&
      -z f-x h&
      0\\
      0&
      0&
      0&
      x^{2}&
      {-x y}&
      y^{2}&
      0&
      0&
      -y f-x g\\
      0&
      x&
      y&
      0&
      z&
      0&
      {-f}&
      {-g}&
      {-h}\\
      0&
      0&
      0&
      0&
      0&
      0&
      x&
      y&
      z\\
      \end{pmatrix},\]
      \[d_3 = \begin{pmatrix}z&
      0&
      0&
      h&
      0&
      0\\
      {-y}&
      {-z}&
      0&
      g&
      h&
      0\\
      x&
      0&
      {-z}&
      {-f}&
      0&
      h\\
      0&
      y&
      0&
      0&
      g&
      0\\
      0&
      x&
      y&
      0&
      {-f}&
      {-g}\\
      0&
      0&
      x&
      0&
      0&
      f\\
      0&
      0&
      0&
      {-y}&
      {-z}&
      0\\
      0&
      0&
      0&
      x&
      0&
      {-z}\\
      0&
      0&
      0&
      0&
      x&
      y\\
      \end{pmatrix},\]
      \[\text{and } d_4 = \begin{pmatrix}{-h}\\
      g\\
      {-f}\\
      z\\
      {-y}\\
      x\\
      \end{pmatrix}.\]
Moreover,
\begin{enumerate}
\item $L$ is $(x,y,z)$-primary.
\item $e(R/L) = 6$. 
\item $\pd\left(\ker(d_{4}^\ast)\right) = \pd\left(R/(x,y,z,f,g,h)\right) - 2$.
\end{enumerate}
\end{prop}

\begin{proof}
It is easy to check that the sequence above forms a complex.  Further, 
one sees that $x^2,y^2,z^2 \in I_1(d_1)$, $x^6,y^6,z^6 \in I_4(d_2)$, $x^5,y^5,z^6 \in I_5(d_3)$ and $I_1(d_4) = (x, y, z, f, g, h)$.  Again by Theorem~\ref{Tres}, the above complex is exact and resolves $R/L$.  Since $\mathrm{ht}(I_1(d_4)) \ge 5$, $L$ is unmixed by Proposition~\ref{Csk}.

Since $\sqrt{L} = (x,y,z)$, $(x,y,z)$ is the unique minimal prime of $I$ and $L$.  Localizing at $\p = (x,y,z)$ we see that $\lambda(R_\p/L_\p) = 6$.  Hence $e(R/L) = 6$ by the associativity formula.  

Clearly $\im(d_4^\ast) = (x, y, z, f, g, h)$.  So $\pd\left(\ker(d_{4}^\ast)\right) = \pd\left(\im(d_{4}^\ast)\right) - 1 = \pd\left(\coker(d_{4}^\ast)\right) - 2 = \pd\left(R/(x,y,z,f,g,h)\right) - 2$.
\end{proof}

We can now define $L_{3,6,p}$ to be any ideal in the polynomial ring by taking $f, g, h$ in the previous definition to be any forms that generate a height $2$ ideal in variables disjoint from $x, y$ and $z$ with $\pd(R/(f, g, h)) = p - 1$.  

\subsection{Construction of $L_{2,6,p}$}

For the last two families of ideals, the resolutions are a bit more complicated.  It takes more work to prove that the canonical modules have large projective dimension.  We also make a specific choice of the forms $f, g, h$ that appear generically at first.  While these constructions likely work for any choice of forms $f, g, h$ that generate an ideal of large projective dimension, it seems more straightforward to prove what we need with a specific choice, such as that from Proposition~\ref{P3gen}.

\begin{prop}
Let $R$ be a polynomial ring over $K$ and let $x,y,t$ be independent linear forms in $R$.  Let $f, g, h \in R_d$ for some $d \ge 1$ such that $\mathrm{ht}(x,y,f,g) = 4$.  Let
     \[
      L = (y^{3},x^{3},x^{2} y^{2},x^{2} y f+x y^{2} g,x^{2}
      f^{2}+x y f g+y^{2} g^{2}+x^{2} y t^{d-1} h).\]
      Then $R/L$ has the following minimal free resolution:
      \[{R} \xleftarrow{d_1} {R}^{5} \xleftarrow{d_2} {R}^{6} \xleftarrow{d_3} {R}^{2} \leftarrow 0\]
where
\[d_1 = \begin{pmatrix}x^{3}&
      y^{3}&
      x^{2} y^{2}&
      x^{2} y f+x y^{2} g&
      x^{2} f^{2}+x y f g+y^{2} g^{2}+x^{2} y t^{d-1} h\\
      \end{pmatrix},\]
\[d_2 = \begin{pmatrix}{-y^{2}}&
      0&
      {-y f}&
      0&
      -f^{2}-y t^{d-1} h&
      0\\
      0&
      {-x^{2}}&
      0&
      {-x g}&
      0&
      {-g^{2}}\\
      x&
      y&
      {-g}&
      {-f}&
      0&
      {-t^{d-1} h}\\
      0&
      0&
      x&
      y&
      {-g}&
      {-f}\\
      0&
      0&
      0&
      0&
      x&
      y\\
      \end{pmatrix},\]
      and
      \[d_3 = \begin{pmatrix}f&
      t^{d-1} h\\
      {-g}&
      0\\
      {-y}&
      f\\
      x&
      {-g}\\
      0&
      {-y}\\
      0&
      x\\
      \end{pmatrix}\]
Moreover,
\begin{enumerate}
\item $L$ is $(x,y)$-primary, and
\item $e(R/L) = 6$.
\end{enumerate}
\end{prop}

\begin{proof}
As in the previous case, we check that we have a complex and that the ideals of minors have the appropriate height.  We have $x^3, y^3 \in I_1(d_1)$, $x^5, y^5 \in I_4(d_2)$ and $x^2, y^2, g^2, f^2 + yt^{d-1}h \in I_2(d_3)$.  By Theorem~\ref{Tres} and Proposition~\ref{Csk}, we have a resolution of $R/L$ and $L$ is unmixed.  Since $\sqrt{L} = (x,y)$, $L$ is $(x,y)$-primary.  

Localizing at $\p = (x,y)$, we see that  $e(R/L) = \lambda(R_\p/L_\p) = 6$.
\end{proof}

\begin{lem}\label{Lpi1} Using the notation from the previous proposition, let $\pi:\im(d_3^\ast) \to R$ be the projection map onto the first coordinate.  Then $\pd(\im(\pi)) = 3$ and
\[\ker(\pi) \iso (x, y, g^{2}, f g, f^{2},t^{d-1} gh).\]
\end{lem}

\begin{proof} Since $\h(f, -g, -y, x) = 4$, we may resolve $\im(\pi) \iso (f, -g, -y, x)$ generically by the Koszul complex on $f, -g -y, x$.  We then get the following commutative diagram with exact rows and columns:
\[
\xymatrix{
&&&R^6 \ar[d]^{\partial_2}&\\
&&R^6 \ar[d]^{d_3^\ast}&R^4 \ar@{_(->}[l]_{\iota} \ar[d]^{\partial_1} &\\
0 \ar[r] & \ker(\pi) \ar[r] & \im(d_3^\ast) \ar[r]^{\pi} \ar[d] & \im(\pi) \ar[r] \ar[d] & 0\\
&&0&0&
}
\]
Here we can represent
\begin{eqnarray*}
d_3^\ast &=& \begin{pmatrix}f&
      {-g}&
      {-y}&
      x&
      0&
      0\\
      t^{d-1} h&
      0&
      f&
      {-g}&
      {-y}&
      x\\
      \end{pmatrix},\\
      \pi &=& \begin{pmatrix} 1&0\end{pmatrix},\\
      \iota &=& \begin{pmatrix}1&
      0&
      0&
      0\\
      0&
      1&
      0&
      0\\
      0&
      0&
      1&
      0\\
      0&
      0&
      0&
      1\\
      0&
      0&
      0&
      0\\
      0&
      0&
      0&
      0\\
      \end{pmatrix},\\
      \partial_1 &=& \begin{pmatrix}f&
      {-g}&
      {-y}&
      x\\
      \end{pmatrix},\\
      \partial_2 &=& \begin{pmatrix}g&
      y&
      0&
      {-x}&
      0&
      0\\
      f&
      0&
      y&
      0&
      {-x}&
      0\\
      0&
      f&
      {-g}&
      0&
      0&
      {-x}\\
      0&
      0&
      0&
      f&
      {-g}&
      {-y}\\
      \end{pmatrix}.
      \end{eqnarray*}

Since the diagram commutes, it follows that 
\begin{eqnarray*}
\ker(\pi) &=& \im(d_3^\ast \circ \iota \circ \partial_2) + \im \begin{pmatrix} 0&0\\-y&x\end{pmatrix} \\
&=& \im \begin{pmatrix}0&
      0&
      0&
      0&
      0&
      0&
      0&
      0\\
      t^{d-1} g h&
      y t^{d-1} h+f^{2}&
      {-f g}&
      -x t^{d-1} h-f g&
      g^{2}&
      -x f+y g&
      -y&
      x\\
      \end{pmatrix} \\
&\iso& (x, y, g^{2}, f g, f^{2},t^{d-1} gh)
\end{eqnarray*}
as claimed.
\end{proof}

\begin{prop} Let $p \ge 3$.  As before, let $f_p = a^p$, $g_p = b^p$ and $h_p = h_p =  a^{p-1}c_1 + a^{p-2}bc_2 + \cdots + b^{p-1}c_p$.  Set 
   \begin{eqnarray*}
   L_{2,6,p} &=& (y^{3},x^{3},x^{2} y^{2},x^{2} y f_p+x y^{2} g_p,x^{2}
      f_p^{2}+x y f_p g_p+y^{2} g_p^{2}+x^{2} y t^{p-1} h_p).\\
      \end{eqnarray*}
            Let $d_3$ be the third map in the minimal free resolution of $R/L_{2,6,p}$.  Then
      \[\pd(\ker(d_3^\ast)) \ge p.\]
      \end{prop}
      
      \begin{proof} Set
            \[J = (x, y, g_p^{2},f_p g_p, f_p^{2}, t^{p-1} g_ph_p).\]
      
      By the previous lemma, we have $\pd(\ker(d_3^\ast)) = \pd(\im(d_3^\ast))  - 1 = \pd(J) - 1 = \pd(R/J) - 2$, as long as $\pd(J) \ge 3$.
        It is clear that $s = t^{p-1}a^{p-1}b^{2p-1} \notin J$ since none of the terms of the generators of $J$ divide $s$.  Let $\p = (a, b, c_1,\ldots,c_p)$.  Since $a^{p-1}b^{p-1} \in (f_p, g_p, h_p):\p$, it is easy to check that $s \in J:\p$.  Hence $R/J$ has an associated prime with height at least $p + 2$, and by Lemma~\ref{Lasspd}, $\pd(R/J) \ge p + 2$.  Therefore, $\pd(\ker(d_3^\ast)) \ge p$.  
\end{proof}

\subsection{Construction of $L_{2,20,p}$}

\begin{prop}\label{Pl220p}
Let $R$ be a polynomial ring over $K$ and let $x,y,t$ be independent linear forms in $R$.  Let $f, g, h \in R_d$ for some $d \ge 1$ such that $\mathrm{ht}(x,y,f,g) = 4$.  Let

\begin{eqnarray*} L &=& (x^{5},y^{5},x^{4} y^{4},x^{4} y^{3} f+x^{3} y^{4}
      g,x^{4} y^{2} f^{2}+x^{3} y^{3} f g+x^{2} y^{4} g^{2},\\
      &&x^{4} y f^{3}+x^{3}y^{2} f^{2} g+x^{2} y^{3} f g^{2}+x y^{4} g^{3},\\
      &&x^{4} f^{4}+x^{3} y f^{3}
      g+x^{2} y^{2} f^{2} g^{2}+x y^{3} f g^{3}+y^{4} g^{4}+x^{4} y^{3} t^{3d - 3} h).
      \end{eqnarray*}
      Then $R/L$ has the following minimal free resolution:
\[{R} \xleftarrow{d_1} {R}^{7} \xleftarrow{d_2} {R}^{10} \xleftarrow{d_3} {R}^{4} \leftarrow 0\]
\[d_1 = \begin{pmatrix}x^{5}&
      y^{5}&
      x^{4} y^{4}&
      x^{4} y^{3} f+x^{3} y^{4} g&
      x^{4} y^{2} f^{2}+x^{3} y^{3} f g+x^{2} y^{4} g^{2}& 
      \ldots
      \end{pmatrix}\]
\[ d_2 = \scalebox{.75}{ $\begin{pmatrix}{-y^{4}}&
      0&
      {-y^{3} f}&
      0&
      {-y^{2} f^{2}}&
      0&
      {-y f^{3}}&
      0&
      -f^{4}-y^{3} t^{3d-3} h&
      0\\
      0&
      {-x^{4}}&
      0&
      {-x^{3} g}&
      0&
      {-x^{2} g^{2}}&
      0&
      {-x g^{3}}&
      0&
      {-g^{4}}\\
      x&
      y&
      {-g}&
      {-f}&
      0&
      0&
      0&
      0&
      0&
      {-t^{3d-3} h}\\
      0&
      0&
      x&
      y&
      {-g}&
      {-f}&
      0&
      0&
      0&
      0\\
      0&
      0&
      0&
      0&
      x&
      y&
      {-g}&
      {-f}&
      0&
      0\\
      0&
      0&
      0&
      0&
      0&
      0&
      x&
      y&
      {-g}&
      {-f}\\
      0&
      0&
      0&
      0&
      0&
      0&
      0&
      0&
      x&
      y\\
      \end{pmatrix}$},\]
            \[d_3 = \begin{pmatrix}f&
      0&
      0&
      t^{3d-3} h\\
      {-g}&
      0&
      0&
      0\\
      {-y}&
      f&
      0&
      0\\
      x&
      {-g}&
      0&
      0\\
      0&
      {-y}&
      f&
      0\\
      0&
      x&
      {-g}&
      0\\
      0&
      0&
      {-y}&
      f\\
      0&
      0&
      x&
      {-g}\\
      0&
      0&
      0&
      {-y}\\
      0&
      0&
      0&
      x\\
      \end{pmatrix}.\]
      Moreover
      \begin{enumerate}
      \item $L$ is $(x,y)$-primary, and
      \item $e(R/L) = 20$.
      \end{enumerate}
\end{prop}      
      
\begin{proof} One checks that $x^5, y^5 \in I_1(d_1)$, $x^{10}, y^{10} \in I_2(d_2)$ and $x^4, y^4, g^4, f^4 - y^3t^{3d-3}h \in I_4(d_3)$.  So the above is a resolution of $R/L$ and $L$ is unmixed.  Since $\sqrt{L} = (x,y)$, $L$ is $(x,y)$-primary.  Localizing at $\p = (x,y)$ we get $e(R/L) = \lambda(R_\p/L_\p) = 20$. 
\end{proof}
      
      \begin{lem} Let $L$ be as in Proposition~\ref{Pl220p}.  Let $\pi:\im(d_3^\ast) \to R^3$ be projection onto the first $3$ rows.  Then $\pd(\im(\pi)) = 3$ and 
      \[\ker(\pi) \iso (x, y, g h t^{3d-3},g^{4},f g^{3},f^{2}
      g^{2},f^{3} g,f^{4}).\]
      \end{lem}
      
      \begin{proof}
      The proof is the exactly the same as Lemma~\ref{Lpi1}.  Since $\h(x, y, f, g) = 4$, we can resolve $\im(\pi)$ generically.   The relevant commutative diagram and matrices are listed below.
      \[
\xymatrix{
&&&0 \ar[d]&\\
&&&R^{3} \ar[d]^{\partial_3}&\\
&&&R^{8} \ar[d]^{\partial_2}&\\
&&&R^{10} \ar[d]^{\partial_1}&\\
&&R^{10} \ar[d]^{d_3^\ast}&R^8 \ar@{_(->}[l]_{\iota} \ar[d]^{\partial_0} &\\
0 \ar[r] & \ker(\pi) \ar[r] & \im(d_3^\ast) \ar[r]^{\pi} \ar[d] & \im(\pi) \ar[r] \ar[d] & 0\\
&&0&0&
}
\]
where
\[
d_3^\ast = \begin{pmatrix}f&
      {-g}&
      {-y}&
      x&
      0&
      0&
      0&
      0&
      0&
      0\\
      0&
      0&
      f&
      {-g}&
      {-y}&
      x&
      0&
      0&
      0&
      0\\
      0&
      0&
      0&
      0&
      f&
      {-g}&
      {-y}&
      x&
      0&
      0\\
      h t^{3d-3}&
      0&
      0&
      0&
      0&
      0&
      f&
      {-g}&
      {-y}&
      x\\
      \end{pmatrix},
                  \]\[
      \pi = \begin{pmatrix} 1&0&0&0\\0&1&0&0\\0&0&1&0\end{pmatrix},
                  \]\[
      \iota = \begin{pmatrix}1&
      0&
      0&
      0&
      0&
      0&
      0&
      0\\
      0&
      1&
      0&
      0&
      0&
      0&
      0&
      0\\
      0&
      0&
      1&
      0&
      0&
      0&
      0&
      0\\
      0&
      0&
      0&
      1&
      0&
      0&
      0&
      0\\
      0&
      0&
      0&
      0&
      1&
      0&
      0&
      0\\
      0&
      0&
      0&
      0&
      0&
      1&
      0&
      0\\
      0&
      0&
      0&
      0&
      0&
      0&
      1&
      0\\
      0&
      0&
      0&
      0&
      0&
      0&
      0&
      1\\
      0&
      0&
      0&
      0&
      0&
      0&
      0&
      0\\
      0&
      0&
      0&
      0&
      0&
      0&
      0&
      0\\
      \end{pmatrix},
                  \]\[
      \partial_0 = \begin{pmatrix}f&
      {-g}&
      {-y}&
      x&
      0&
      0&
      0&
      0\\
      0&
      0&
      f&
      {-g}&
      {-y}&
      x&
      0&
      0\\
      0&
      0&
      0&
      0&
      f&
      {-g}&
      {-y}&
      x\\
      \end{pmatrix},
                  \]\[
      \partial_1 = \begin{pmatrix}0&
      0&
      0&
      {-x}&
      g&
      y^{3}&
      0&
      0&
      0&
      0\\
      0&
      0&
      0&
      {-y}&
      f&
      0&
      {-y^{3}}&
      {-x y^{2}}&
      {-x^{2} y}&
      x^{3}\\
      0&
      0&
      {-x}&
      g&
      0&
      y^{2} f&
      y^{2} g&
      0&
      0&
      0\\
      0&
      0&
      {-y}&
      f&
      0&
      0&
      0&
      {-y^{2} g}&
      {-x y g}&
      x^{2} g\\
      0&
      {-x}&
      g&
      0&
      0&
      y f^{2}&
      y f g&
      y g^{2}&
      0&
      0\\
      0&
      {-y}&
      f&
      0&
      0&
      0&
      0&
      0&
      {-y g^{2}}&
      x g^{2}\\
      x&
      g&
      0&
      0&
      0&
      f^{3}&
      f^{2} g&
      f g^{2}&
      g^{3}&
      0\\
      y&
      f&
      0&
      0&
      0&
      0&
      0&
      0&
      0&
      g^{3}\\
      \end{pmatrix},
            \]\[
      \partial_2 = \begin{pmatrix}{-f^{3}}&
      {-f^{2} g}&
      0&
      {-f g^{2}}&
      0&
      {-g^{3}}&
      0&
      0\\
      y f^{2}&
      y f g&
      0&
      y g^{2}&
      0&
      0&
      0&
      {-g^{3}}\\
      y^{2} f&
      y^{2} g&
      0&
      0&
      0&
      0&
      y g^{2}&
      {-x g^{2}}\\
      y^{3}&
      0&
      0&
      0&
      y^{2} g&
      0&
      x y g&
      {-x^{2} g}\\
      0&
      0&
      y^{3}&
      0&
      x y^{2}&
      0&
      x^{2} y&
      {-x^{3}}\\
      x&
      0&
      {-g}&
      0&
      0&
      0&
      0&
      0\\
      {-y}&
      x&
      f&
      0&
      {-g}&
      0&
      0&
      0\\
      0&
      {-y}&
      0&
      x&
      f&
      0&
      {-g}&
      0\\
      0&
      0&
      0&
      {-y}&
      0&
      x&
      f&
      g\\
      0&
      0&
      0&
      0&
      0&
      y&
      0&
      f\\
      \end{pmatrix},
      \]\[
      \partial_3 = \begin{pmatrix}g&
      0&
      0\\
      {-f}&
      g&
      0\\
      x&
      0&
      0\\
      0&
      {-f}&
      g\\
      {-y}&
      x&
      0\\
      0&
      0&
      {-f}\\
      0&
      {-y}&
      x\\
      0&
      0&
      y\\
      \end{pmatrix}.
\]
The proof is finished by computing $\ker(\pi) = \im(d_3^\ast \circ \iota \circ \partial_1) + \im \begin{pmatrix}0&0\\0&0\\0&0\\-y&x\end{pmatrix}$.
      \end{proof}
      
\begin{prop} Let $p \ge 3$.  As before, let $f_p = a^p$, $g_p = b^p$ and $h_p = h_p =  a^{p-1}c_1 + a^{p-2}bc_2 + \cdots + b^{p-1}c_p$.  Set 
\[L_{2,20,p} = (x^{5},y^{5},x^{4} y^{4},x^{4} y^{2} f_p^{2}+x^{3} y^{3} f_p g_p+x^{2} y^{4} g_p^{2},x^{4} y f_p^{3}+x^{3}
      y^{2} f_p^{2} g_p+x^{2} y^{3} f_p g_p^{2}+x y^{4} g_p^{3},\]
      \[x^{4} y^{3} f+x^{3} y^{4}
      g,x^{4} f_p^{4}+x^{3} y f_p^{3}
      g_p+x^{2} y^{2} f_p^{2} g_p^{2}+x y^{3} f_p g_p^{3}+y^{4} g_p^{4}+x^{4} y^{3} t^{3p - 3} h_p).\]
            Let $d_3$ be the third map in the minimal free resolution of $R/L_{2,20,p}$.  Then
      \[\pd(\ker(d_3^\ast)) \ge p.\]
      \end{prop}
            
      \begin{proof} By the previous lemma, we have $\pd(\ker(d_3^\ast)) = \pd(\im(d_3^\ast)) - 1 = \pd(J) - 1$,  where
      \[J = (x, y, g_p h_p t^{3p-3},g^{4},f_p g_p^{3},f_p^{2}
      g_p^{2},f_p^{3} g_p,f_p^{4}).\] 
      One again checks that $s = t^{3p - 3} a^{p-1} b^{4p - 1} \in J:(a, b, c_1,\ldots, c_p) \setminus J$.  It follows that $\pd(\ker(d_3^\ast)) \ge p$ by Lemma~\ref{Lasspd}.
      \end{proof}

\section{Examples and Questions}\label{Seq}

In this section, we first explicitly construct one of the primary ideals from the main theorem.  We also discuss the problem of classifying ideals satisfying Serre's $(S_2)$ property.

In general we do not write down the primary ideals with large projective dimension as they have a large number of generators.  For instance, if one takes $f, g, h$ to be the forms from Proposition~\ref{P3gen}, the minimal number of generators of $I_{2,4,p}$ were computed using Macaulay2 \cite{M2} and are listed in the following table.\\

\begin{center}
\begin{tabular}{r|l|l|l|l|l|l|l|l|l|l|l}
$p$:&5&6&7&8&9&10&11&12&13&14&15\\
\hline 
$\mu(I_{2,4,p})$:&9&12&17&25&38&59&93&148&237&381&614\\
\end{tabular}\\
\end{center}

We explicitly compute $I_{2,4,6}$ in the following example.

\begin{eg} Here we give a homogeneous $(x,y)$-primary ideal $I$ in a polynomial ring $R$ with $e(R/I) = 4$ and $\pd(R/I) = 6$.  Let $R = K[a,b,c,d,e,x,y]$.  First set
\[f = a^3, \qquad g = b^3, \qquad h = a^2c+abd+b^2e.\]
Then $\pd(R/(f,g,h)) = 5$, $\mathrm{ht}(f,g,h) = 2$ and $\pd(R/(x,y,f,g,h)) = 7$.  By the above construction, if we set
\[L = (x^3, x^2y, xy^2, y^3, fx^2 + gxy + hy^2) \quad \text{and} \quad I = (x^3,y^3):L.\]
then
\[\pd(R/I) = \pd(R/(x,y,f,g,h)) - 1 = 7 - 1 = 6,\]
by Proposition~\ref{Pl25p}.
  In this case, $I$ has the following $12$ minimal generators:
\[\begin{matrix}x^{3},
      x^2 y, 
      x y^2, 
      y^{3}\\
      a^{2} c x y+a b d x y+b^{2} e x y-b^{3} y^{2}\\
      a c d x^{2}+b d^{2} x^{2}-b c e x^{2}-a e^{2} x y-a^{2} d y^{2}+a b e y^{2}\\
      b c^{2} x^{2}-a d^{2} x y+a c e x y-b d e x y-a b c y^{2}+b^{2} d y^{2}\\
      a c^{2} x^{2}+b c d x^{2}-a d e x y-b e^{2} x y-a^{2} c y^{2}+b^{2} e y^{2}\\
      b^{2} c x^{2}+a^{2} d x y+a b e x y-a b^{2} y^{2}\\
      a b c x^{2}+b^{2} d x^{2}+a^{2} e x y-a^{2} b y^{2}\\
      a^{2} c x^{2}+a b d x^{2}+b^{2} e x^{2}-a^{3} y^{2}\\
      b^{3} x^{2}-a^{3} x y\\
      \end{matrix}\]
      and Betti table
      \[\begin{tabular}{r|ccccccc}
      &0&1&2&3&4&5&6\\
      \hline
     \text{0:}&1&\text{-}&\text
      {-}&\text{-}&\text{-}&\text{-}&\text{-}\\\text{1:}&\text{-}&\text{-}&\text{-}&\text{-}&\text{-}&\text{-}&\text{-}\\\text{2:}&\text{-}&4&3&\text{-}&\text{-}&\text{-}&\text{-}\\\text{3:}&\text{-}&\text{-}&\text{-}&\text{-}&\text{-}&\text{-}&\text{-}\\\text{4:}&\text{-}&8&26&33&21&7&1\\\end{tabular}\]
\end{eg}

Finally, we remark that while all of the ideals we construct satisfy Serre's $(S_1)$ condition, almost none of the ideals we construct are $(S_2)$, even on the punctured spectrum.  Let $I$ be unmixed of height $h$ and suppose $\x = x_1,\ldots, x_h \in I$ is a regular sequence with $I \neq (\x)$.  Set $L = (\x):I$.  By \cite[Theorem 4.1]{Schenzel}, $S/I$ is $(S_2)$ if and only if $H^{d-1}_\m(R/L) = 0$, where $d = \dim(R/I)$ and $\m$ is the graded maximal ideal of $R$.  By local duality, this is equivalent to $\ext_R^{h+1}(R/L,R) = 0$.  As each of the four ideals from Section~\ref{S4ideals} satisfy $\pd(R/L) = \h(L) + 1$, $\ext_R^{\h(L)+1}(R/L, R)$ is the cokernel of a nonsurjective map and hence nonzero itself.  So none of the ideals we construct are $(S_2)$.

We could weaken the question to ask which ideals are $(S_2)$ on the punctured spectrum, i.e. $(S_2)$ scheme theoretically in projective space.  Such ideals $I$ satisfy
\[\depth(R_\p/I_\p) \ge \min_{\substack{\p \in \spec(R)\\\p \neq \m}} \{2,\dim(R_\p/I_\p)\}.\]
One checks that $R/I$ is $(S_2)$ on the punctured spectrum if and only if $\lambda(\ext_R^{h + 1}(R/L,R)) < \infty$.  In the construction of $L = L_{2,5,p}$, if we take the forms $f, g, h$ to be variables and work in the ring $R = K[x, y, f, g, h]$, then the third differential in the resolution of $L$ is $d_3 = \begin{pmatrix} x&y&f&g&h \end{pmatrix}^\mathsf{T}$.  Then $\ext_R^3(R/L,R) = \coker(d_3^\ast) \iso R/(x,y,f,g,h)$, which is finite length by construction.  Hence the linked ideal 
\[I = (x^3, y^3):L = (x^3,x^2y, xy^2, y^3, x^{2} g-x y h, x y f-y^{2} g, x^{2} f-y^{2} h)\]
is $(S_2)$ on the punctured spectrum.  Since $\pd(R/L) = 3$, it follows that $R/I$ is actually Cohen-Macaulay on the punctured spectrum.  This ideal appears as one case in the Manolache's characterization of Cohen-Macaulay structures \cite[Theorem 1]{Manolache}.  However, if we take $f, g, h$ to be forms that generate an ideal of projective dimension at least $4$, then $\ext_R^3(R/L, R)$ will no longer be finite length and $I$ will no longer be scheme-theoretically $(S_2)$.  So we pose the following question:
\begin{qst} Is there a finite classification of homogeneous unmixed or primary ideals of a given height and multiplicity that satisfy Serre's $(S_2)$ condition? Is there a classification for such ideals that are $(S_2)$ on the punctured spectrum?
\end{qst}
\noindent Three of the ideals in Proposition~\ref{Pe2} (cases $(1)$, $(2)$ and $(5)$) are Cohen-Macaulay, while the other two are not even $(S_2)$ globally.  Since $(S_2)$ implies $(S_1)$, we have a complete characterization of height $2$ and multiplicity $2$ ideals that are globally $(S_2)$.  The other two ideals can be chosen to define an $(S_2)$ projective scheme.
For instance, if $R = K[a,b,x,y]$ and $I = (x,y)^2 + (ax+by)$, then $R_\p/I_\p$ is $(S_2)$ for all nonmaximal primes $\p$.  Note that $(S_2)$ on the punctured spectrum does not imply $(S_1)$, so the latter question above remains unclear even in height $2$ and multiplicity $2$.  The case of ideals of height $2$ and multiplicity $3$ appears to be completely open.

\section*{acknowledgements}
This work was done while the authors were at MSRI.  We thank the organizers of the special year in commutative algebra and the rest of the personnel at MSRI for making this possible.  All of the examples and resolutions in this paper were computing with Macaulay2 \cite{M2}.  We also thank Manoj Kummini for useful several conversations.  The first author was supported by NSF grant number 1259142.

\bibliographystyle{amsplain}
\bibliography{bib}

\end{document}